\documentclass{amsart}

\usepackage{amsmath,amssymb,amsfonts,amsthm}
\usepackage{color}
\newcommand{\I}{\mathrm{i}}
\newcommand{\D}{\mathrm {d}}

\newcommand{\beq}{\begin{equation}}
\newcommand{\eeq}{\end{equation}}

\DeclareMathOperator{\diverg}{div}

\DeclareMathOperator{\ran}{ran}
\DeclareMathOperator{\meas}{meas}

\newtheorem{lemma}{Lemma}[section]
\newtheorem{theorem}[lemma]{Theorem}

\newtheorem{corollary}[lemma]{Corollary}
\theoremstyle{definition}
\newtheorem{definition}[lemma]{Definition}
\theoremstyle{definition}
\newtheorem{remark}[lemma]{Remark}
\theoremstyle{definition}
\newtheorem{example}[lemma]{Example}

\newcommand{\N}{{\mathbb N }}
\newcommand{\R}{{\mathbb R}}
\newcommand{\C}{{\mathbb C}}

\newcommand{\s}{{\mathcal S}}
\def\si{{\sigma}}
\def\vsi{{\varsigma}}
\newcommand{\e}{{\varepsilon }}
\newcommand{\ie}{{\sl i.e.\/ }}
\newcommand{\cf}{{\sl cf.\/ }}
\newcommand{\eg}{{\sl e.g.\/}}

\def\d{{\partial}}

\def\({\left(}
\def\){\right)}
\def\<{\left\langle}
\def\>{\right\rangle}
\def\O{\mathcal O}
\def\res{{\rm res}}

\def\be{{\mathbf E}}

\def\B{\mathcal B}

\def\E{{\rm e}}
\def\i{{\rm i}}

\def\ol{\overline}
\def\YYY{\mathcal Y}

\newcommand{\Id}[1]{{\rm I\kern-2pt I_{#1}}}
\renewcommand{\hbar}{{\displaystyle\bar{\phantom{x}}\kern-6pt h}}

\numberwithin{equation}{section}

\renewcommand{\Im}{\mathop{\mathrm{Im}}}

\begin{document}

%%%%%%%%%%%%%%%%%%%%%%%%%%%%%%%%%%%%%%%%%%%%%%%%%%%%%%%%%%%%%%%%%%%%%

\title[Pulse interaction in a NLS]{Interaction of modulated pulses in the 
nonlinear Schr\"odinger equation with periodic potential} 
\author[J.~Giannoulis]{Johannes Giannoulis}
\author[A.~Mielke]{Alexander Mielke}
\author[C.~Sparber]{Christof Sparber}
\address[J.~Giannoulis]{Zentrum Mathematik, 
Technische Universit\"at M\"unchen\\ Boltz\-mann\-stra{\ss}e 3\\D-85747 Garching b.\ M\"unchen\\Germany}
\email{giannoulis@ma.tum.de}
\address[A.~Mielke]{Weierstra{\ss}-Institut f\"ur Angewandte Analysis und 
Stochastik\\ Mohrenstra{\ss}e 39\\ 10117 Berlin  
\and Institut f\"ur Mathematik, Humboldt-Uni\-ver\-si\-t\"at zu Berlin\\ 
Rudower Chaussee 25\\12489 Berlin\\Germany}
\email{mielke@wias-berlin.de}
\address[C.~Sparber]{Wolfgang Pauli Institute Vienna \& Faculty of Mathematics, Vienna University, Nordbergstra\ss e 15, A-1090 Vienna, Austria}
\email{christof.sparber@univie.ac.at}

\begin{abstract} We consider a cubic nonlinear Schr\"odinger equation with
  periodic potential. In a semiclassical scaling the nonlinear interaction of
  modulated pulses concentrated in one or several Bloch bands is studied. The
  notion of closed mode systems is introduced which allows for the
  rigorous derivation of a finite system of amplitude equations describing the
  macroscopic dynamics of these pulses.
\end{abstract}

 
\thanks{This work has been partially supported by the DFG Priority Program
  1095 {\it ``Analysis, Modeling and Simulation of Multiscale Problems''}
  under contract number Mi459/3-3. The third author has been supported by the
  APART research grant of the Austrian academy of sciences.}
\maketitle
\begin{center}
version: April 27, 2007
\end{center}

%%%%%%%%%%%%%%%%%%%%%%%%%%%%%%%%%%%%%%%%%%%%%%%%%%%%%%%%%%%%%%%%%%%%%%%

\section{Introduction and main result}

In this work we study the asymptotic behavior for $0<\e \ll 1$ of
the following nonlinear Schr\" odinger equation (NLS)
\begin{equation}
\label{nls}
\I \e \partial _t u^\e =  -\frac{\e^2}{2}\Delta u^\e +
V_{\Gamma}\left(\frac{x}{\e}\right)u^\e + 
\e \kappa \, |u^\e|^2 u^\e,\quad x\in \R^d, t\in \R, 
\end{equation}
governing the dynamics of a wave field $u^\e(t,\cdot)\in L^2(\R^d)$.  Here, $\kappa \in
\R$ and the potential $V_\Gamma= V_\Gamma(y)\in \R$ is assumed to be 
smooth and
\emph{periodic} with respect to some \emph{regular lattice} $\Gamma \simeq
\mathbb Z^d$, generated by a given basis $\{\zeta_1,\dots,\zeta_d\}$, $\zeta_l
\in \R^d$, \ie
\begin{equation}
\label{eq:Vper}
V_{\Gamma}(y + \gamma) = V_{\Gamma}(y), \quad \forall \, y \in \R^d,
\gamma \in \Gamma\equiv  \textstyle
\big\{\gamma=\sum_{l=1}^d \gamma_l
  \zeta_l\in \R^d:  \: \gamma_l\in \mathbb Z \big\}. 
\end{equation}
It is well known that if $\kappa < 0$ the solution of
\eqref{nls} in general does not exist for all times, \ie finite-time
\emph{blow-ups} 
may occur, \cf \cite{SuSu}.

The equation \eqref{nls} can be seen as a simplified model of the one
considered in \cite{CMS}.  There the main motivation was to study, from a
\emph{semiclassical} point of view, the dynamics of a Bose-Einstein condensate
in an optical lattice, described by $V_\Gamma$, \cf \cite{ChNi, DFK, KMPS}
for more details. To this end a rescaling of the appearing physical parameters
yields an equation similarly to \eqref{nls}, but with an additional
non-periodic confining potential, which we shall neglect in the following.
The parameter $\e \ll 1$ then describes the \emph{microscopic/macroscopic
  scale ratio}. The main assumption for the analysis presented in \cite{CMS}
has been that the initial data $u^\e(0,\cdot)$ is 
supposed to be of \emph{WKB
  type} and in particular it 
has to be concentrated in a single (isolated)
\emph{Bloch band} $E_\ell(k)\in \R$.  These energy bands describe the spectral
subspaces corresponding to the \emph{periodic Hamiltonian operator}
\begin{equation}
\label{ham}
H_{\rm per}^\e: =  - \frac{\e^2}{2} \Delta +
V_{\Gamma}\left(\frac{x}{\e}\right),
\end{equation}
\cf Section \ref{secbloch} below for more details. In the linear case similar
WKB approximations have been established earlier in \cite{BLP, GRT}, yielding
an approximate macroscopic description (\ie on time-- and length--scales of order one) 
of the highly oscillatory solution to
\eqref{nls}.  However, the question concerning a generalization of the results in
\cite{CMS}, in particular to the case of
multiple bands, has been open so far.  Here we 
will answer this question for initial data which correspond to a sum of 
modulated \emph{plane waves}.

In order to derive an approximate macroscopic description we shall proceed by
a \emph{two scale expansion method} similar to that in \cite{CMS}.
To this end a detailed understanding of the influence of the nonlinearity is
crucial. Indeed we will show that the 
solution to \eqref{nls} can be approximated 
(in a suitably scaled Sobolev space) via 
\begin{equation}\label{approxintro}
u^\e (t,x) \sim \sum_{m=1}^M a_m(t,x) \chi_{\ell_m} \left(\frac{x}{\e}; k_m \right) \E^{\I (k_m \cdot x - t E_{\ell_m}(k_m) ) /\e)} 
+ \O(\e),
\end{equation} 
for $M \in \N$, where the set $\{ (k_m, \ell_m)\: : \: m=1,...,M\} $ is
assumed to form a \emph{closed mode system}, see Definition \ref{defclosed},´
of sufficient high order $\Lambda$. As we shall see $\Lambda$ will depend on the spatial 
dimension $d$. The
amplitudes $a_m$ are the (local-in-time) solutions to the nonlinear system of 
amplitude equations 
\begin{equation}
\label{systemIntro}
\I \partial_t a_m +  \I\vartheta_m \cdot\nabla_x a_m
=  \sum^M_{ {p,q,r =1} : \atop {\Sigma(\mu_p, \mu_q, \mu_r) =\Sigma(\mu_m)} }
\kappa_{(p,q,r,m)} \, a_{p} \, \ol a_{q} \,  a_{r}
,\quad m=1,\ldots,M.
\end{equation} 
The above system describes a so-called \emph{four-wave interaction}, also known (most prominently in 
laser physics and nonlinear optics) as 
\emph{four-wave mixing}, \cf \cite{BMS, CMBSKP, HM}. 
By \eqref{approxintro} we allow for nonlinear interactions 
within the same band but also consider interactions of different bands. In
particular, energy or mass transfer between different bands is expected due to 
the presence of the nonlinearity on the right hand side of
\eqref{systemIntro}.
To our knowledge this phenomenon has never been studied rigorously in the
context of Schr\"odinger type equations.  

One should note that the concept of wave mixing is strongly linked to
plane waves as considered above.  Indeed, if one allows for more general
phases (which is possible in the case of a single pulse \cite{CMS}), a rigorous understanding, even in
much simpler cases, is lacking and in particular one can not expect the
nonlinear interaction to be maintained on macroscopic time-scales in general.
On the other hand we could, without any problems, allow for (smooth)
\emph{higher order nonlinearities} as considered in \cite{CMS}. This however would
result into a much more involved resonance-structure for the nonlinear interactions
and in order to keep our presentation simple we restrict
ourselves to the cubic case.
 
Before going into more details, let us briefly mention the following mathematically rigorous works
which, besides \cite{CMS}, are most closely related to ours: In \cite{Sp} the same
equation as \eqref{nls} is considered but in a slightly different scaling.
Similarly, a nonlinear Schr\"odinger type model is derived in \cite{BGTU} from
a semilinear wave equation with periodic coefficients and in \cite{GiMi} from
an underlying oscillator chain model. Concerning nonlinear wave interactions,
there exist several results (mostly three-wave mixing) in the context of
strictly hyperbolic systems, see, \eg, \cite{HMR, JMR, MR, Ra}, and in the
case of microscopically discrete dynamical systems, \cf \cite{Gia}, \cite{GHM} 
and the references given therein. 

The paper is organized as follows: In Section \ref{secmode} we 
introduce the basic notions needed in the following, \ie Bloch bands, mode
systems and the \emph{closure condition}. We also discuss several
illustrative examples of wave mixing. Then in Section \ref{formal} the formal derivation of
the approximate solution is given in detail. In Section \ref{just} the
obtained formal asymptotic description is rigorously established and 
our main theorem is stated for a closed mode system
of order $\Lambda=2N+1$, with $N>d/2$. In Section \ref{highres} 
it is shown that this condition can be relaxed though by introducing the 
concept of \emph{weak closure}. 
Finally, in Appendix A, we 
discuss in more detail the underlying Hamiltonian structure of the amplitude equations 
\eqref{systemIntro}.

\section{Mode systems and resonances} \label{secmode}

For our work it is essential to study the spectral properties of 
$$
H_{\rm per} = - \frac{1}{2} \Delta_y + V_\Gamma(y),
$$ since they 
will basically determine the fast degrees of freedom in our model. 

\subsection{Bloch's spectral problem} \label{secbloch}

In what follows we denote by $Y$ the centered \emph{fundamental domain} of the
lattice $\Gamma$, \ie
\begin{equation}\label{eq:Y}
Y:= \left\{\gamma \in \R^d: \ \gamma=\sum_{l=1}^d \gamma_l \zeta_l, 
\ \gamma_l\in \Big[-\frac{1}{2}, \, \frac{1}{2}\Big] \right\}.
\end{equation}
By $\YYY$ we denote the $d$-dimensional torus $\R^d_{\!/\Gamma}$, which
is obtained also from $Y$ by identifying opposite faces. Note that writing
$H^s(\YYY)$ then automatically includes periodicity conditions.
Moreover, $Y^*$ denotes the corresponding basic cell of the dual lattice
$\Gamma^*$. Equipped with periodic boundary conditions $\YYY^*$ is usually called the \emph{Brillouin zone}
and hence we shall denote it by $\mathcal B\equiv \YYY^*$.

Next, consider the so-called \emph{Bloch eigenvalue problem} \cite{Bl}, \ie
the following spectral problem on $\YYY$:  
\begin{equation}\label{bloch}
H_{\Gamma}(k) \chi_\ell (y; k) = \, E_\ell (k)\chi_\ell (y; k), \quad  \text{$k \in \mathcal B$, $\ell \in \N$},
\end{equation}
where $E_\ell(k)\in \R$ and $\chi_\ell(y)\equiv \chi_\ell(y;k)$ denote the $\ell$-th eigenvalue and eigenvector
of the \emph{shifted Hamiltonian operator}
\begin{equation*}
H_{\Gamma}(k):= \E^{-\I k \cdot y} H_{\rm per} \, \E^{\I k \cdot y} = \frac{1}{2} \, \left(-\I \nabla_y + k \right)^2+
V_{\Gamma}\left (y\right) .
\end{equation*}
Let us recall some well known facts for this eigenvalue problem, \cf
\cite{B, BLP, ReSi, Wi}. 
Since $V_\Gamma$ is smooth and periodic, we get that, 
for every fixed $k\in \mathcal B$, $H_\Gamma(k)$ is self-adjoint on
$L^2(\YYY)$ with domain
$H^2(\YYY)$ and compact resolvent.  
Hence the spectrum of $H_\Gamma(k)$ is given by
\begin{equation*}
\text{spec} (H_\Gamma(k))= \{
E_\ell(k)\ ;\ \ell \in \N\} \subset \R. 
\end{equation*}
One can order the eigenvalues $E_\ell(k)$ according to their
magnitude and multiplicity such that 
\[
E_1(k)\leq\ldots\leq E_\ell(k)\leq E_{\ell+1}(k)\leq \dots
\]
Moreover every $E_\ell(k)$ is periodic w.r.t. $\Gamma^*$ and it holds that
$E_\ell(k)=E_\ell(-k)$. The set $\{E_\ell (k); \, k \in \mathcal B\}$ is
called the $\ell$th \emph{energy band}.  The associated
eigenfunctions, the \emph{Bloch functions}, $\chi_\ell(y; k)$ form (for every
fixed $k\in\mathcal B$) an orthonormal basis in $L^2(\YYY)$. 
We choose the usual normalization such that
\begin{equation*}
\left <\chi_{\ell_1}(k), \chi_{{\ell_2}}( k) \right>_{L^2(\YYY)}\equiv 
\int_\YYY \overline{\chi_{\ell_1}}(y; k)\, \chi_{\ell_2} (y; k) \, \D y =
\delta_{\ell_1, \ell_2},\quad \ell_1 ,\, \ell_2 \in\N. 
\end{equation*} 
Concerning the dependence on $k\in \mathcal B$, it has been shown, 
\cf \cite{ReSi}, that for any $\ell\in \N$ 
there exists a closed subset $\mathcal U\subset \mathcal B$ such that $E_\ell(k)$ 
is analytic in $\mathcal O:=\mathcal B \backslash \mathcal U$. 
Similarly, the eigenfunctions 
$\chi_\ell$ are found to be analytic and periodic in $k$, 
for all $k \in \mathcal O$. Moreover it holds that 
\begin{equation*}
E_{\ell -1}(k) < E_\ell(k) < E_{\ell +1}(k),\quad \forall \, k \in \mathcal O.
\end{equation*}
If this condition is satisfied for all $k\in \mathcal B$ then $E_\ell(k)$ is said to be an \emph{isolated Bloch band}. 
Finally we remark that  
\[
\meas \mathcal U = \meas \, \{ k\in\mathcal B\ | \ E_{\ell_1}(k)=E_{\ell_2}(k), \ \ell_1 \not = \ell_2 \}=0.
\] 
In this set of measure zero one encounters so-called \emph{band crossings}. 
The elements of this set are characterized by the fact that $E_\ell(k)$ is only 
Lipschitz continuous and hence the \emph{group velocity} $\vartheta:=\nabla_k E_\ell(k)$ does \emph{not} exist.

\subsection{Resonances and closed mode systems} 

Our goal is to derive an approximate description of our model for $\e \ll 1$. To
this end we shall first introduce several definitions needed to do so.

\begin{definition}\label{def2.1} 
For $k \in \mathcal B$ and $\ell \in \N$ we call $\mu =(k,\ell)$ a {\em mode} 
and $\mathcal M:= \B\times\N$ the \emph{set of all modes}.\\
(i) The \emph{graph of all modes} is given by
\[
\mathcal G = \big \{(k, E_\ell (k)): (k,\ell) \in \mathcal M \big \}\subset \mathcal B
\times \R.  
\]
(ii) Given a finite set $\mathcal S=\{ \mu_1,\dots, \mu_M: M \in \N
\}\subset \mathcal M$ of modes, we call $\mathcal S^\Lambda$ the (ordered)
\emph{mode system of size $\Lambda \in \N$ generated by $\mathcal S$}.\\
(iii) We further introduce $\Sigma: \mathcal S^\Lambda  \rightarrow
\mathcal B \times \R$ by 
\begin{equation*}
  \Sigma(\mu_{m_1},\dots, \mu_{m_\Lambda}):=
\left (\sideset{}{^*}\sum_{\lambda =1}^{\Lambda} ({-}1)^{\lambda+1}
  k_{m_\lambda} \, , \, 
    \sum_{\lambda=1}^{\Lambda} ({-}1)^{\lambda+1} E_{\ell_{m_\lambda}} 
(k_{m_\lambda})\right),
\end{equation*}
where $\sum^*$ denotes summation modulo $\Gamma^*$, and we write
$\mathcal G_\mathcal S^{(\Lambda)} := \Sigma (\mathcal S^\Lambda)$ for the corresponding
graphs.
\end{definition} 

In the following the mapping $\Sigma$ will describe the possible nonlinear
interaction of modes (in every order of $\e$). Note that $\mathcal G_\mathcal S^{(1)} \subset \mathcal G$ and moreover,
for any $\Lambda \in \N$, it holds $\mathcal G_\mathcal S^{(\Lambda)} \subset
\mathcal G_\mathcal S^{(\Lambda+2)}$. Since we are dealing with a cubic 
nonlinearity, we shall see that indeed $\Lambda \in \N$ takes only \emph{odd} values. 

\begin{definition} \label{defclosed}
Given a finite set of modes $\mathcal S$ and a subset $\mathcal T$ of
$\mathcal M$.\\
(i) An element $(\mu_{m_1},\dots, \mu_{m_\Lambda}) \in \mathcal
S^\Lambda$ is called \emph{resonant of order $\Lambda$ to $\mathcal T$}, if
\[
\Sigma(\mu_{m_1},\dots, \mu_{m_\Lambda}) \in  \mathcal G^{(1)}_{\mathcal T}. 
%=\Sigma (\mathcal T).
\]
(ii) We say that $\mathcal S$ {\em is closed of order $\Lambda$}, if 
the group velocity $\nabla_k E_\ell(k)$ exists 
for all $\mu=(k,\ell)\in\mathcal S$ and
\begin{equation}\label{closed} 
\mathcal G_\mathcal S^{(\Lambda)} \cap \left(\mathcal G\setminus
    \mathcal G_\mathcal S^{(1)}\right) = \emptyset.
\end{equation}
\end{definition}
 
We infer from the above definition that a single mode $\mu$ is always resonant of order $\Lambda=1$ 
to itself. 
Note that for any $\Lambda \geq 2$ however, we obtain 
$$
\Sigma: (\mu_{m_1},\dots, \mu_{m_\Lambda}) \mapsto \sigma:=(k, E) 
\in \mathcal B\times \R,$$ 
which does not necessarily imply 
$\sigma\in \mathcal G^{(1)}_{\mathcal S} $. 
As for the condition \eqref{closed}, it is equivalent to saying that for all
$(\mu_{m_1},\dots, \mu_{m_\Lambda}) \in \s^\Lambda$ it holds: Either
$\Sigma(\mu_{m_1},\dots, \mu_{m_\Lambda}) \in (\mathcal B \times \R) \setminus
\mathcal G $, or else there exists a $\mu \in \mathcal S$ such that $\Sigma
(\mu_{m_1},\dots, \mu_{m_\Lambda}) = \Sigma (\mu)\in \mathcal G_\s^{(1)}$.  
The latter obviously means that $(\mu_{m_1},\dots, \mu_{m_\Lambda})$ 
is resonant of order $\Lambda$ to $\mathcal S$. We illustrate these concepts in the 
examples in Section \ref{examples}.

The condition that $\nabla_k E_\ell(k)$ has to exist implies that we can not deal with mode systems 
including band crossings, \ie which include a $\mu=(k,\ell)$ such that $k \in \mathcal U$. 
It is known that the higher the dimension $d$ and the higher the band index $\ell$, the more likely 
one encounters such crossings. Thus, in terms of practical use one can expect our analysis to be 
restricted to cases where only a few bands with low energies are taken into account.

\begin{remark} Also note that condition \eqref{closed} is not equivalent to
$\mathcal G_\mathcal S^{(\Lambda)} \cap \mathcal G \subset \mathcal G_\mathcal S^{(1)}$, since due to
multiple eigenvalues $E_\ell (k)$, we may have modes $\mu \in \mathcal S$ and
$\tilde \mu \not \in \mathcal S$ with $\Sigma (\mu) =\Sigma (\tilde \mu)$.
\end{remark}

\subsection{The system of amplitude equations} \label{structure}

In Section \ref{formal} we shall  derive a system of amplitude equations 
describing the macroscopic dynamics of our nonlinear wave interactions. Before going into the details of the derivation 
let us first discuss the general structure of the obtained system (see also Appendix A). 
In order to keep our presentation simple, we shall from now on 
consider only the case of \emph{non-degenerated eigenvalues} $E_\ell(k)$.

Let $\mathcal S$ be a given finite set of modes $\{\mu_1,\dots , \mu_M\}$
which is closed of sufficient high order $\Lambda$.
Then the equations obtained for the modulating amplitudes 
$a_m(t,x)\in \C$, for $m \in \{1,\dots, M \}$, are as in \eqref{systemIntro}, \ie 
\begin{equation*} 
\I \partial_t a_m +  \I\vartheta_m \cdot\nabla_x a_m
=  \sum^M_{ {p,q,r =1} : \atop {\Sigma(\mu_p, \mu_q, \mu_r) =\Sigma(\mu_m)} } \kappa_{(p,q,r,m)} \, a_{p} \, \ol a_{q} \,  a_{r}.
\end{equation*}
Here $\vartheta_m \in \R^d$ is the group velocity corresponding to a 
given mode $\mu_m =(k_m,\ell_m) \in \s$, \ie $\vartheta_m:= \nabla_k E_{\ell_m}(k_m)$, and we further denote by
\begin{equation*}
\kappa_{(p,q,r,m)}:= \kappa
\int_{\mathcal Y} \chi_{\ell_{p}}(y) \overline \chi_{\ell_q}(y)
\chi_{\ell_r}(y) \overline \chi_{\ell_m}(y) \, \D y   
\end{equation*}
the \emph{effective coupling constant} $\kappa_{(p,q,r,m)}$, which in general is complex valued.  

As in \eqref{nls}, 
the nonlinearity in the amplitude equations is again cubic. It 
takes into account the sum over all $p,q,r= 1,\dots, M$, for 
which the \emph{resonance condition of order $\Lambda=3$} holds. 
Explicitly, this is equivalent to the following two conditions
\begin{equation} \label{rcon}
k_p - k_q + k_r = \,  k_m, \quad 
E_{\ell_p}(k_p) - E_{\ell_q}(k_q) + E_{\ell_r}(k_r) = \, E_{\ell_m}(k_m),
\end{equation}
where the first equation has to be 
understood as a summation modulo $\Gamma^*$. 
The equations \eqref{rcon} describe a so-called \emph{four-wave interaction},
\ie three modes $\mu_p, \mu_q, \mu_r$ being in resonance with a fourth one
$\mu_m \in \s$.  Clearly, the nonlinear structure on the right-hand side of
\eqref{systemIntro} is such that energy transfer between different modes can occur.
In other words, even if initially some of the amplitudes $a_m(0,\cdot)$ are zero, 
they will not remain so in general for times $|t|\not =0$ (see also Subsection \ref{4wave}
below).

Concerning existence and uniqueness of solutions to the above given amplitude equations, 
we only need a local-in-time result. 
\begin{lemma} \label{lem1}
For any initial data $(a_1(0,\cdot),\ldots,a_M(0,\cdot)) \in H^S(\R^d)^M$, 
with $S > d/2$, 
the system \eqref{systemIntro} admits a unique solution 
$$
(a_1,\ldots,a_M) \in C^0([0,T); H^S(\R^d))^M \cap C^1((0,T); H^{S-1}(\R^d))^M,$$ 
up to some (finite) time $T >0$.
\end{lemma}
\begin{proof} 
Since the left hand side of \eqref{systemIntro} only generates 
translations by a constant velocity $\vartheta_m \in \R^d $ 
and thus conserves the $L^2(\R^d)$ norm of each $a_m (t,x)$, 
the assertion of the lemma follows by a standard fixed point argument.
\end{proof}

\begin{remark} Discarding for a moment the nonlinear term in \eqref{systemIntro}
  we want to point out that the remaining linear transport part is much simpler
  than in the case of the full WKB type approximation, as discussed in \cite{CMS}
  (for a single mode only). 
In particular we do not run into any problems
  due to \emph{caustics}, since the phase functions 
  $$
  \varphi_m(t,x)=k_m \cdot x - t E_{\ell_m}(k_m)
  $$
  are globally defined. Also note that in the present work the so-called \emph{Berry term} vanishes,  
  since we do not take into account additional non-periodic potentials, \cf \cite{CMS, PST}. 
\end{remark}

\subsection{Examples}\label{examples} In order to illustrate the abstract concepts defined in the former subsections we 
shall in the following consider several particular examples of nonlinear wave
interactions appearing in our model.

\subsubsection{Four-wave interaction for three pulses} \label{4wave} 

In this example we shall restrict ourselves to a set of modes 
$\{\mu_m=(k_m,\ell_m)\,:\, m=1,2,3\}$ which is closed of order $\Lambda=3$, at least, 
and which satisfies the following resonance conditions 
\begin{equation}\label{rconex}
k_1-k_2+k_1 = k_3, \quad E_{\ell_1}(k_1) - E_{\ell_2}(k_2) + E_{\ell_1}(k_1) =
E_{\ell_3}(k_3). 
\end{equation} 
(Again the first summation has to be understood modulo $\Gamma^*$.) 
This yields the following set of amplitude equations 
\begin{equation}\label{amsys}
\left \{
\begin{aligned}
  \I \partial_t a_{1} + \I \vartheta_1 \cdot\nabla_x a_1
  = & \,     W_1(a)a_1 +  2  \overline \kappa_{(1,2,1,3)} 
       \, \overline a_1 \, a_2\,  a_3,\\
  \I \partial_t a_2 + \I \vartheta_2 \cdot\nabla_x a_2
  = & \,    W_2(a)a_2 +  \kappa_{(1,2,1,3)} \,  a^2_1 \, \overline a_3,\\
  \I \partial_t a_3 + \I \vartheta_3 \cdot\nabla_x a_3 = & \,
  W_3(a)a_3 + \kappa_{(1,2,1,3)} \, a^2_1 \, \overline a_2,
\end{aligned}
\right.
\end{equation}
where we shortly denote 
\begin{equation*} 
W_m(a):= \kappa_{(m,m,m,m)}
|a_m|^2 + 2 \sum_{j=1,2,3: \atop j \not =m} \kappa_{(m,j,j,m)} |a_j|^2\ \in
\R.  
\end{equation*} 
We consequently expect
the solution of \eqref{nls} to be asymptotically described by 
\begin{equation*} u^\e(t,x)
\sim \sum _{m=1}^3 a_m(t,x) \chi_{\ell_m} \left(\frac{x}{\e}; k_m \right)
\E^{\I (k_m \cdot x - E_{\ell_m} (k_m)t)/\e} + \O(\e), \quad \ell_m \in \N.
\end{equation*} 
It is easily seen that the nonlinearities $W_m(a)a_m$  
can not transfer energy from one band to the other
as they are homogeneous in the respective $a_m(t,x)$ and that all of the appearing $\kappa_{(m,j,j,m)}$,
$m,j=1,2,3$ are indeed \emph{real valued}.  
What is more important though is the fact that the above given amplitude
system \eqref{amsys} includes an \emph{invariant sub-system}.  Namely, if
initially $a_1(0,\cdot)=0$, it remains so for all $t\in \R$ and thus the above
given system simplifies to
\begin{equation*}
\left \{
\begin{aligned}
\I \partial_t a_2 + \I \vartheta_2 \cdot\nabla_x a_2
= & \,    \left(\kappa_{(2,2,2,2)}|a_2|^2+ 2 \kappa_{(2,3,3,2)}|a_3|^2\right) a_2, \\
\I \partial_t a_3 + \I \vartheta_3 \cdot\nabla_x a_3
= & \,  \left(\kappa_{(3,3,3,3)}|a_3|^2+ 2 \kappa_{(2,3,3,2)}|a_2|^2\right) a_3 .
\end{aligned}
\right.
\end{equation*}
However, such a decoupling does not exist if initially $a_{2}(0,\cdot)=0$, 
since this amplitude will be generated during the course of time by the 
remaining two others. Analogously $a_{3}$ is generated by interaction 
of $a_1$ and $a_2$.

More formally we can also infer this fact from the closure condition, since any
possible combination of $k_2$ and $k_3$ via the mapping $\Sigma$ (with
$\Lambda=3$) yields either $k_2, k_3$, or any other value (like
$2k_2-k_3$) which is not in $\mathcal G_\s^{(1)}$ by assumption (recall that the system
of modes $\mu_1,\mu_2,\mu_3$ is assumed to be closed of order $\Lambda=3$).
However, if one aims to follow the same argument for, \eg, $k_1$ and $k_2$,
the first equation in \eqref{rconex} shows that this sub-system is not closed.

\subsubsection{The case of several pulses within one Bloch band} 

This is a particular situation where we keep the band index $\ell \in \N$
fixed and thus only consider the interaction of several pulses within a single
Bloch band.  Hence the solution to \eqref{nls} takes the asymptotic form
\begin{equation*}
  u^\e(t,x) \sim \sum _{m=1}^M a_m(t,x)  \chi_\ell \left(\frac{x}{\e}; k_m \right) \E^{\I (k_m \cdot x - E_\ell (k_m)t)/\e} 
  + \O(\e) , \quad \ell \in \N.
\end{equation*}
In particular we can expect this description to be correct in cases where
the Bloch band $E_\ell(k)$ admits only a moderate variation $\Delta_\ell :=
\text{var}_{k\in \mathcal B} \ E_\ell(k)$ and is well separated from the rest
of the spectrum of $H_\Gamma(k)$ by a sufficiently large gap, \ie
\[
\min\{ |E_\ell(k)-E_n(k)| \::\: n\in \N,\ n \neq \ell \}
\gg \Delta_\ell >0.
\] 
It is easy to show then that a four-wave interaction can \emph{always} 
be realized within such a band. Choose $k_\text{max}$
and $k_\text{min}$ such that $E_\ell(k_\text{min} ) \leq E_\ell(k) \leq
E_\ell(k_\text{max})$ for all $k \in \mathcal B$. We are looking for a triple
$(k_1,k_2,k_3)$ satisfying \eqref{rconex} with $\ell_j=\ell$, $j=1,2,3$. 
To this end we note that $k_3=2k_1-k_2$, modulo $\Gamma^*$, and define
\[
e(k_1,k_2)\equiv 2E_\ell(k_1)-E_\ell(k_2)-E_\ell(2k_1{-} k_2).
\] 
Then we have 
$ e(k_\text{max},k_\text{min})>0>e(k_\text{min},k_\text{max})$ and by a simple
application of the intermediate-value theorem, we easily find $k_1,k_2$ with
$k_1\neq k_2$ and $e(k_1,k_2)=0$. 

\subsubsection{The case of a single pulse decomposed into several bands} 

This is a second particular case, where we expect an asymptotic description
for solutions to \eqref{nls} given by
\begin{equation*} 
  u^\e(t,x) \sim \sum
  _{\ell=1}^L a_\ell(t,x) \chi_\ell \left(\frac{x}{\e}; k_0 \right) \E^{\I
    (k_0 \cdot x - E_\ell (k_0)t)/\e} + \O(\e) ,  
\quad k_0 \in \R, 
\end{equation*} 
where $k_0$ is some given wave vector. One should have the following
intuition: Given an initial plane wave of the form
\begin{equation*} 
u_{\rm in}^\e (x) = f(x)
  \E^{ \I k_0 \cdot x/ \e}, 
\end{equation*} 
one decomposes the slowly varying
\emph{macroscopic} amplitude $f(x)$ into a sum of terms, each of which is
concentrated on a single Bloch band. This is possible due to Bloch's theorem,
which ensures that $L^2(\R^d)= \bigoplus_{\ell =1}^\infty \mathcal H_\ell$,
where $\mathcal H_\ell$ are the so-called band spaces. Strictly speaking though, one would require 
countably many terms in the decomposition, which we can take into account here. 
However for any practical purposes (and if $f(x)$ is sufficiently smooth and rapidly decaying) 
only the first few Bloch bands need to be taken into account as all higher bands 
give negligible contributions, \cf \cite{HJMS}. 
We shall not go into more details here on the
precise definition of $\mathcal H_\ell$ etc. but refer to \cite{B, BLP, ReSi, Wi} 
for more details on these classical results (see also \cite{HJMS} for a
numerical approach).

Concerning the possible generation of resonant modes, it is clear that in this case 
the first condition in \eqref{rcon} is 
trivially fulfilled, since we are only dealing with a single wave vector $k_0$. Thus, it is merely a question 
on the precise structure of the bands $E_\ell(k_0)$ 
whether one can indeed expect resonances.

\section{Formal derivation of the approximate solution}
\label{formal} 

In the following we consider a finite set $\mathcal S$ of modes 
which is closed of order $\Lambda = 2 N +1$, for some $N \in \N$, 
to be determined later. 

For solutions of \eqref{nls} we seek an asymptotic two-scale expansion of
the following form  
\begin{equation}
\label{wkb1}
u_N^\e (t,x):= 
\ \sum_{n=0}^N \e^n \, v_n
\left(t,x,\frac{t}{\e},\frac{x}{\e}\right), 
\end{equation}
where 
\begin{equation}
\label{wkb2}
 v_n(t,x,\tau,y):=
\sum_{\sigma \in \mathcal G_\mathcal S^{(2n+1)}} A_{n, \sigma}(t,x,y)\be_{\sigma}\left(\tau,y\right),
\end{equation}
with $\be_{\sigma }(\tau,y):=\E^{\i \sigma \cdot (y, -\tau )}$, for 
$\sigma \in \mathcal G_\mathcal S^{(2n+1)} \equiv \Sigma (\mathcal S^{2n+1})$. 
Note that in the most simple case, where $n=0$, we get that $\sigma = (k, E_{\ell} (k))$, with $(k,\ell)\in \mathcal S$, which yields 
$\be_{\sigma }(\tau,y)= \E^{ \I (k\cdot y - E_\ell(k) \tau)}$, a simple plane wave. Moreover, if $\sigma_j \in \mathcal G_\s^{(2n_j+1)}$, 
for $j=1,2,3$, this yields
\begin{equation} \label{relation}
\be_{\sigma_1} \overline{\be}_{\sigma_2} \be_{\sigma_3} = \be_{\sigma_1-\sigma_2+\sigma_3 }, 
\end{equation}
with $\sigma_1-\sigma_2+\sigma_3 \in \mathcal G_\s^{(2(n_1+n_2+n_3)+3)}$. Of
course we also impose that
\begin{equation*}
A_{n, \sigma}(\cdot,\cdot,y + \gamma) = A_{n, \sigma}(\cdot,\cdot,y), \quad \forall \, y
\in \R^d,  
\, \gamma \in \Gamma.
\end{equation*}

\subsection{The general strategy}\label{ss3.1}

As already said before, we shall only consider the case of simple eigenvalues
$E_\ell(k)$, for simplicity.  Plugging the ansatz \eqref{wkb1}, \eqref{wkb2}
into \eqref{nls} and expanding in powers of $\e$, we formally obtain
\begin{equation*}
\I \e \partial _t u_N^\e -  H_{\rm per}^\e   u_N^\e-  \e \kappa |u_N^\e|^2 u_N^\e  = 
\sum_{n=0}^N \e^n X_n +\res(u_N^\e),
\end{equation*}
with the residual 
\begin{equation}\label{res}
\res(u_N^\e) =\sum_{n =N+1}^{3N+1}\e^n X_n.
\end{equation}
Introducing, for any (general) $\sigma \equiv (k,E)\in \mathcal B\times
\R$, the operators
\begin{equation} \label{l}
L_0^{\sigma}:= E - H_{\Gamma}(k), \quad 
L_1^{\sigma}:=\i \partial_t + \i k \cdot \nabla_x +\diverg_x\nabla_y,
\quad
L_2:=\frac12\, \Delta_x
\end{equation}
this yields 
\begin{align}\label{x0}
X_0:= 
\sum_{\sigma \in \mathcal G_\s^{(1)}}(L_0^{\sigma }A_{0,\sigma} ) \be_\sigma,
\end{align}
and also
\begin{equation}\label{x1}
\begin{split}
X_1:= & \ 
\sum_{\sigma \in \mathcal G_\s^{(3)}}(L_0^{{\sigma}}A_{1,\sigma}) \be_{\sigma} +
\sum_{\sigma \in \mathcal G_\s^{(1)}}(L_1^{\sigma}A_{0,\sigma} )\be_{\sigma}\\
 & \ - \kappa  \sum_{\sigma_j \in G_\s^{(1)}: \atop {j=1,2,3}}
A_{0,\sigma_1}\ol A_{0,\sigma_2} A_{0,\sigma_3} \be_{\sigma_1 - \sigma_2 + \sigma_3}
\end{split}
\end{equation}
where we have used the relation \eqref{relation}. In general we get for $n = 2,\dots, 3N+1$, 
\begin{equation}\label{xn}
\begin{split}
X_n:= &  \sum_{\sigma \in \mathcal G_\s^{(2n+1)}} (L_0^{\sigma }A_{n,\sigma} ) \be_\sigma + 
\sum_{\sigma \in \mathcal G_\s^{(2n-1)}} (L_1^{\sigma} A_{n-1,\sigma}) \be_\sigma + \sum_{\sigma \in \mathcal G_\s^{(2n-3)}} L_2 A_{n-2,\sigma} \be_\sigma \\
&  - \kappa \sum_{n_1+n_2+n_3 = n-1}  \sum_{\sigma_j \in G_\s^{(2n_j+1)}:  \atop {j=1,2,3}}
A_{n_1,\sigma_1}\ol A_{n_2,\sigma_2} A_{n_3,\sigma_3} \be_{\sigma_1 - \sigma_2 + \sigma_3}.
\end{split}
\end{equation}
Here one should note that $A_{n, \sigma } (t,x,y) \equiv 0$, 
for all $n \geq N+1$, by assumption.

Now we shall subsequently construct $A_{n,\si}$ such that $X_n
\equiv 0$ for $n=0,\ldots,N$.  
To this end we have to compare \emph{equal}  
coefficients of $\be_\si$. 
We consequently obtain, from \eqref{x0}--\eqref{xn}, equations of the form
\begin{equation} \label{allg}
L_0^\sigma A_{n,\sigma} = F_{n,\si}, \quad \si \in \mathcal G_\s^{(2n+1)},
\end{equation}
where the r.h.s.\ $F_{n,\si}$ can be determined from the coefficients $(
A_{m,\si})_{  \si \in \mathcal G_\s^{(2m+1)}}$ for  
$m = 0,1,\dots, n-1$. More precisely
\begin{equation} \label{fnsi}
\begin{split}
F_{n,\si}:= & \, 
- L_1^\si A_{n-1,\si} - L_2 A_{n-2, \si} \\
& \, +\kappa \sum_{{n_1+n_2+n_3=n-1}}
\sum_{{\sigma_{j}\in \mathcal G_\s^{(2n_j +1)}:}\atop{\si_1-\si_2+\si_3 = \si}} A_{n_1,{\si_1}} \ol A_{n_2,{\si_2}} A_{n_3,{\si_3}}.
\end{split}
\end{equation}
Note that here the summation index has changed in comparison to \eqref{xn}. 
To proceed further we need to distinguish two possible cases:

{\bf Case I.}  On the one hand, for $\si \in \mathcal G_\s^{(2n+1)}
\setminus \mathcal G_\s^{(1)}$ the closure condition up to order $2n+1$ implies
invertibility of $L_0^\si$, \ie $(L_0^\si)^{-1} \in \text{Lin}
(L^2(\YYY),H^2(\YYY))$ and we obtain
\begin{equation} \label{noninv}
A_{n, \si}(t,x,y)  = (L_0^\si)^{-1} F_{n,\si}(t,x,y).
\end{equation}
The corresponding modes are called \emph{non-resonant}.

{\bf Case II.}  
On the other hand, if $\sigma \in \mathcal G_\s^{(1)}$, then $L^\sigma_0$
has a nontrivial kernel. In order to distinguish this case more prominently from the one above, we 
shall from now on use the notation $\vsi \equiv \sigma \in \mathcal G_\s^{(1)}$, which also 
characterizes the basic resonant modes $\mu\in{\mathcal S}$ via $\vsi = \Sigma(\mu)$. 

Using the orthogonal projections $\mathbb P_\vsi$ onto this kernel the
necessary and sufficient solvability condition for \eqref{allg} in
this case is then given by 
\begin{equation} \label{sovcon}
\mathbb P_\vsi F_{n,\vsi}(t,x,y)=0,
\end{equation}
which yields (recall that $\text{dim} (\text{ran}\,\mathbb P_\vsi) =1$, by
assumption) 
\begin{equation} \label{sovcon1}
\left<  \chi_\vsi , \, F_{n,\vsi} \right>_{L^2(\YYY)}= 0 , \quad \text{for
  $0 \neq \chi_\vsi \in \text{ran}\,\mathbb P_\vsi$}. 
\end{equation}
Under this condition, we consequently obtain
\begin{equation}\label{ampli}
A_{n, \vsi} (t,x,y)= a_{n,\vsi}(t,x) \chi_\vsi(y) + A^\perp_{n,\vsi} (t,x,y)
\end{equation}
with
\begin{equation} \label{perp}
A^\perp_{n, \vsi}:=(L_0^\vsi)^{-1} (1- \mathbb P_\vsi) F_{n,\vsi}.
\end{equation}
Note that here $a_{n,\vsi}$ is still undetermined. However the condition \eqref{sovcon1}  
provides a partial differential equation for $a_{n-1,
  \varsigma}$, the so far undetermined part of the previous step.

\begin{remark}\label{r3.1}
  In the case of non-simple eigenvalues $E_\ell(k)$, \ie $\text{dim}
  (\text{ran}\,\mathbb P_\vsi) = R >1$, we can simply use a 
  smooth orthonormal basis $\{ \chi_{\vsi, r} \}_{r=1}^R$ of $\ran \mathbb
  P_\vsi$ and generalize the above given formulas \eqref{sovcon},
  \eqref{ampli} accordingly. 
\end{remark} 

\subsection{Explicit calculations}\label{ss3.2}

In the following we shall determine the approximate solution in more detail,
by following the above described strategy.

{$n=0$}: We need to solve $X_0 \equiv 0$ and immediately 
note that in this case, Case I above is obsolete. Then, in Case II, equation \eqref{x0} implies  
\begin{equation*}
L_0^\vsi A_{0,\vsi}=0,\quad \vsi \in \mathcal G_\s^{(1)}
\end{equation*}
and thus \eqref{ampli} simplifies to 
\begin{equation} \label{a0}
A_{0,\vsi} = a_{0,\vsi}(t,x) \chi_\vsi(y),
\end{equation}
with $a_{0,\vsi}$ still to be determined.

{$n=1$}: In the next step we have to solve $X_1 \equiv 0$. In Case I, 
\ie for $\si \in \mathcal G_\s^{(3)} \setminus \mathcal G_\s^{(1)}$, the equations
\eqref{x1} and \eqref{noninv} imply
\begin{equation}\label{A1}
\begin{split} 
A_{1, \si} = & \ \kappa \, (L_0^\si)^{-1}\Big(  \sum_{{\varsigma_1,\varsigma_2,\varsigma_3 \in \mathcal G_\s^{(1)}}\atop 
 \si = \varsigma_1 - \varsigma_2 + \vsi_3}
A_{0,\varsigma_1}\ol A_{0,\vsi_2} A_{0,\varsigma_3} \Big)\\
= & \ \kappa  \sum_{{\varsigma_1,\varsigma_2,\varsigma_3 \in \mathcal G_\s^{(1)}:}\atop \si = \varsigma_1 - \varsigma_2 + \varsigma_3} 
a_{0,\vsi_1} \overline {a}_{0,\vsi_2} a_{0,\vsi_3} \, (L_0^\si)^{-1} 
\left(\chi_{\vsi_1} \overline {\chi}_{\vsi_2} \chi_{\vsi_3} \right),
\end{split}
\end{equation}
where for the second equality we simply insert \eqref{a0}. 
We proceed with Case II: To this end the solvability condition \eqref{sovcon}
for $\vsi \in \mathcal G_\s^{(1)}$ allows us to determine the so far still unknown
$a_{0,\vsi}$, obtained before. By \eqref{sovcon1}, this yields
\begin{equation}\label{sx1}
\int_{\YYY} 
\overline \chi_{\vsi}(y) \, F_{1,\vsi}(t,x,y) 
\, \D y=0,
\end{equation}
where 
\begin{equation}\label{F1vsi}
F_{1,\vsi}=  L_1^\vsi A_{0, \vsi}  - \kappa
 \sum_{{\vsi_1,\vsi_2,\vsi_3 \in \mathcal G_\s^{(1)}}: \atop {\vsi = \varsigma_1 - \varsigma_2 + \varsigma_3}} A_{0, \vsi_1} \overline{A}_{0, \vsi_2}A_{0, \vsi_3}.
\end{equation}
From the definition of $L_1^\varsigma$ in \eqref{l} and using the following basic identity  
\begin{equation*}
  \left<  \chi_\ell , \,  (-\i\nabla_y+k) \chi_\ell  \right>_{L^2(\YYY)}
  =  \nabla_k E_\ell(k),
 \end{equation*}
a straightforward calculation shows, \cf the appendix of \cite{CMS}, that \eqref{sx1} can be written as 
\begin{equation}\label{trans1}
\partial_t a_{0,\vsi} + \vartheta_\vsi \cdot\nabla_x a_{0,\vsi} + 
\sum_{{\vsi_1,\vsi_2,\vsi_3 \in \mathcal G_\s^{(1)}}: 
\atop {\vsi = \varsigma_1 - \varsigma_2 + \varsigma_3}}
\i \kappa_{(\vsi_1,\vsi_2,\vsi_3,\vsi)} 
  a_{0, \vsi_1} \overline{a}_{0, \vsi_2}a_{0, \vsi_3} =0,
\end{equation}
where we denote
\begin{equation*}
\kappa_{(\vsi_1,\vsi_2,\vsi_3,\vsi)}:= \kappa 
\int_{\YYY}
 \chi_{\vsi_1}(y) \overline {\chi}_{\vsi_2}(y)
\chi_{\vsi_3}(y) \overline {\chi}_{\vsi}(y) \, \D y .   
\end{equation*}
Since any $\vsi \in \mathcal G_\s^{(1)}$ corresponds to a unique $m=1,\dots, M$, 
via $\vsi = \Sigma(\mu_m)=(k_m, E_{\ell_m}(k_m))$, 
with $\mu_m =(k_m,\ell_m) \in \mathcal S$, we can shortly write
\[
a_{0,\vsi}(t,x) \equiv a_m(t,x) , \quad \chi_\vsi(y) \equiv
\chi_{\ell_m}(y;k_m), \quad \vartheta_\vsi \equiv \vartheta_m.
\] 
Hence the amplitude equations \eqref{trans1} can be equivalently written 
in the form \eqref{systemIntro} used before. 

In summary we have now fully determined the expressions \eqref{a0} and \eqref{A1}, 
and from \eqref{ampli} we finally get that
\begin{align*} 
A_{1, \vsi} (t,x,y)= a_{1,\vsi}(t,x) \chi_\vsi(y) 
+ (L_0^\vsi)^{-1} (1- \mathbb P_\vsi) F_{1,\vsi}(t,x,y),
\end{align*}
where again the coefficients $a_{1,\vsi}$ are still arbitrary and have to be 
determined by the solvability condition for $n=2$. Since this yields an
initial value problem for $a_{1,\vsi}$ (see below) we are free to choose its 
value at time $t=0$. For simplicity we shall put $a_{1,\vsi}(0,\cdot)=0$.
 
{\bf $n\geq 2$}: From here we proceed inductively, as described above, by
solving $X_n\equiv0$. The only difference that occurs is that in Case II 
the coefficients $a_{n,\vsi}$ do not solve a nonlinear
initial value problem, but rather a system of \emph{linear, inhomogeneous}
transport equations. Indeed, lengthy calculations show that
\begin{equation} \label{intrans}
\begin{split}
  \partial_t a_{n,\vsi} + \vartheta_\vsi \cdot\nabla_x a_{n,\vsi} +
  \sum_{{\tilde\vsi, \vsi_1,\vsi_2 \in \mathcal G_\s^{(1)}}: \atop {\tilde\vsi-\varsigma_1+\vsi_2
      = \vsi}}
 2 \i \kappa_{(\tilde\vsi, \vsi_1,\vsi_2, \vsi)} 
    a_{n,\tilde\vsi} \overline{a}_{0,{\vsi_1}}  a_{0,\vsi_2}\quad\text{ } \\
 + \sum_{{\vsi_1,\tilde\vsi, \vsi_2 \in \mathcal G_\s^{(1)}}: \atop {\vsi_1-\tilde\vsi+\vsi_2 = \vsi}}
 \i \kappa_{(\vsi_1, \tilde\vsi, \vsi_2, \vsi)}  
   a_{0,\vsi_1} \overline{a}_{n,\tilde\vsi} a_{0,\vsi_2} + \i \Theta_{n} = 0,
\end{split}
\end{equation}
with source term 
\begin{align*}
  \Theta_n :=& \ -\mathbb P_\vsi \, \big(L_1^\vsi A_{n,\vsi}^\perp+L_2 A_{n-1,\vsi}\big) \\
  & \ +\kappa\,\mathbb P_\vsi \Big(\sum_{{\tilde\vsi, \vsi_1,\vsi_2 \in \mathcal G_\s^{(1)}}:
    \atop {\tilde\vsi-\varsigma_1+\vsi_2 = \vsi}} 2 A^\perp_{n,\tilde\vsi}
  \overline{A}_{0,{\vsi_1}} A_{0,\vsi_2} + \sum_{{\vsi_1, \tilde\vsi, \vsi_2 \in
      \mathcal G_\s^{(1)}}: \atop {\vsi_1-\tilde\vsi+\vsi_2 = \vsi}}
  A_{0,\vsi_1} \overline{A}^\perp_{n,\tilde\vsi}  A_{0,\vsi_2}\Big )  \\
  & \ +\kappa\,\mathbb P_\vsi \Big(\sum_{{\vsi_1,\vsi_2 \in \mathcal G_\s^{(1)}\not \ni
      \si}: \atop {\si -\varsigma_1+\vsi_2 = \vsi}} 2 A_{n,\si}
  \overline{A}_{0,{\vsi_1}} A_{0,\vsi_2} + \sum_{{\vsi_1, \vsi_2 \in
      \mathcal G_\s^{(1)}\not \ni \si}: \atop {\vsi_1-\si+\vsi_2 = \vsi}}
  A_{0,\vsi_1} \overline{A}_{n,\si}  A_{0,\vsi_2}\Big )  \\
  & \ + \kappa\, \mathbb P_\vsi \Big(\sum_{{n_1+n_2+n_3=n:\atop n_j\leq n-1}}
  \sum_{{\sigma_{j}\in \mathcal G_\s^{(2n_j +1)}}:\atop{\si_1-\si_2+\si_3 = \vsi}}
  A_{n_1,{\si_1}} \ol A_{n_2,{\si_2}} A_{n_3,{\si_3}} \Big).
\end{align*}
Again we shall put $a_{n, \vsi} (0,\cdot)= 0$ for simplicity, 
since we are free to
choose the initial values for \eqref{intrans}. Finally, we note that, in order
to satisfy $X_N \equiv 0$ (\ie in the last step), we do not need
to determine the corresponding $a_{N,\vsi}$ and thus  we can impose
$a_{N,\vsi}(t,\cdot)\equiv 0$.

\section{Justification of the amplitude equations}  
\label{just}

In order to obtain our main result, Theorem \ref{th1}, we have to 
justify the above given formal calculations rigorously. 
In particular we need a nonlinear stability result on our approximation.

\subsection{Estimates on the approximate solution and on the residual} 
\label{ss4.1}

In the above section we derived an approximate solution $u_N^\e$ which 
admits an asymptotic expansion of the following form
\begin{equation}
\label{uNfinal} 
\begin{split}
u_N^\e (t,x) = & \ \sum_{\vsi \in G_\s^{(1)}} a_{0,\vsi}(t,x) \chi_\vsi \left(\frac {x}{\e} \right) \be_\vsi \left(\frac {t}{\e},\frac{x}{\e} \right)\\
& \ + \sum_{n=1}^N      \e^n  \sum_{\vsi \in G_\s^{(1)}} \left( a_{n,\vsi}(t,x) \chi_\vsi \left(\frac {x}{\e} \right) + 
A_{n,\vsi}^\perp \left(t,x,\frac{x}{\e} \right) \right) \be_\vsi \left(\frac {t}{\e},\frac{x}{\e}\right) \\
 & \ +   \sum_{n=1}^N   \e^n \sum_{\si \in G_\s^{(2n+1)}\setminus G_\s^{(1)}} A_{n,\si}\left(t,x,\frac{x}{\e} \right)   \be_\si \left(\frac {t}{\e},\frac{x}{\e}\right)
\end{split}
\end{equation}
where the first two lines on the r.h.s.\ include only resonant modes while the
last one takes into account the generated non-resonant terms.  To this end we
required the appearing finite set of modes to be closed of order $\Lambda =
2N{+}1$.

To proceed further let us, for $s\in[0,\infty)$, introduce the
scaled Sobolev spaces 
\begin{equation*}
H^s_\e :=\left\{ f^\e \in L^2(\R^d)\ ;\  {\|f^\e\|}_{H^s_\e}  <\infty\right\},
\end{equation*}
with
\begin{equation*}
{\| f^\e \|}^2_{H^s_\e}:= \int_{\R^d} {(1+  |\e p \, |^2)}^s \, | \widehat f
(p) |^2 \, \D p.
\end{equation*}
Here $\widehat f$ denotes the usual Fourier transform of $f$ in $L^2(\R^d)$. 
Note that in $H_\e^s$ the following Gagliardo--Nirenberg inequality holds
\begin{equation}\label{GN}
\forall\,s> \frac{d}{2} \ \, \exists \, C_\infty>0: \quad 
{\|\, f \, \|}_{L^\infty} \leq C_\infty {\|\, f \, \|}_{H^s} \leq  
C_\infty {\e^{-d/2}}
{\|\, f \, \|}_{H^s_\e},
\end{equation}
where the factor $\e^{-d/2}$ is easily obtained by scaling. 
In the following lemma we collect the a-priori estimates on
$u_N^\e$ needed in order to prove our main result.

\begin{lemma} \label{lem2-neu}
For $d,N\in\N$ and $s\in[0,\infty)$ let $K=\max\{0,s+\frac{d}{2}-2\}$ and 
$S>N+s+\frac{d}{2}$, or, if $d=1$, $S\ge N+s+1$. Moreover, assume 
that $\partial^\alpha V_\Gamma \in L^\infty (\YYY)$, for all $|\alpha| \leq K$,
and let  
$(a_1,\ldots,a_M) \in C^0( [0,T),H^S(\R^d))^M$ 
be a solution of the amplitude equations \eqref{systemIntro}. 

Then, the approximate solution $u^\e_N$, given in \eqref{uNfinal}, 
satisfies the following estimates. 
For each $T_*\in(0,T)$ there exist positive constants
$C_a,C_b,C_r>0$, such that, for $\e\in(0,1)$, $|\alpha| \leq s$, 
and $t \in [0, T_*]$, it holds
\begin{align*}
 { \|\, (\e\d_x)^\alpha u^\e _{N} (t,\cdot) \, \| }_{L^\infty} \leq C_a,\quad
 { \|\, u^\e _{N} (t,\cdot) \, \| }_{H^{s}_\e} \leq C_b,\quad
 { \| \, \res(u_N^\e) (t,\cdot) \, \| }_{H^{s}_\e} \leq C_r \e^{N+1}.
\end{align*}
\end{lemma}
\begin{proof} 
Since $u^\e _{N}$ is given by \eqref{uNfinal}, 
in order to prove the estimates of the lemma we need to establish the
regularity of $A_{n,\sigma}$ for $n=0,\ldots,N$, $\sigma\in \mathcal G_\s^{(2n+1)}$,
which have been calculated formally in Section \ref{formal}.
Under the assumptions $(a_1,\ldots,a_M) \in C^0( [0,T),H^S(\R^d))^M$ and
$\partial^\alpha V_\Gamma \in L^\infty (\YYY)$, for $|\alpha| \leq K$,
we shall first prove that it holds
\begin{equation}\label{Ansigma}  
A_{n,\sigma} \in C^0([0,T),H^{S-n}(\R^d,H^{K+2}(\YYY))). 
\end{equation}
Indeed, by \eqref{a0}, $A_{0,\vsi}(t,x,y) = a_{0,\vsi}(t,x) \chi_\vsi(y)$,
where $a_{0,\vsi}$ solves \eqref{trans1} 
and $\chi_\vsi$ is given by \eqref{bloch}, and we get immediately 
\eqref{Ansigma} for $n=0$ (\cf Lemma \ref{lem1} and  Lemma \ref{lem3} below).

From here it follows, on the one hand, that 
for $\sigma\in \mathcal G_\s^{(3)}\setminus \mathcal G_\s^{(1)}$,
by \eqref{allg} and \eqref{A1},  
$F_{1,\sigma}
=L_0^\sigma A_{1,\sigma} \in C^0([0,T),H^{S}(\R^d,H^{\tilde s}(\YYY)))$
with  $\tilde s=K+2$ and hence, by \eqref{fnsi}, 
$A_{1,\sigma} \in C^0([0,T),H^{S-1}(\R^d,H^{\tilde s}(\YYY)))$.
On the other hand, for $\vsi\in \mathcal G_\s^{(1)}$ we obtain by \eqref{F1vsi}
$F_{1,\vsi}\in C^0([0,T),H^{S-1}(\R^d,H^{\tilde s}(\YYY)))$,
and thus, by 
\eqref{perp},
$A_{1,\vsi}^\perp \in C^0([0,T),H^{S-1}(\R^d,H^{\tilde s}(\YYY)))$.
Moreover, 
$a_{1,\vsi}\chi_\vsi\in C^0([0,T),H^{S-1}(\R^d,H^{\tilde s}(\YYY)))$,
since $a_{1,\vsi}$ solves the linear inhomogeneous transport equation 
\eqref{intrans} (for $n=1$)
with initial condition $a_{1,\vsi}(0,\cdot)=0$, coefficients 
in  $C^0([0,T),H^{S}(\R^d))$,
and source term $\Theta_1\in C^0([0,T),H^{S-2}(\R^d))$.
Hence, by \eqref{ampli}, we obtain \eqref{Ansigma} for $n=1$. 
Proceeding inductively, we obtain \eqref{Ansigma} 
for all $n=0,\ldots,N$ via the above given steps.

Having established \eqref{Ansigma}, we aim to show the first estimate of the lemma. To this end 
we have to guarantee that for each $n=0,\ldots,N$, $\sigma\in \mathcal G_\s^{(2n+1)}$,  
there exists a $C>0$, such that  
\begin{equation*}
 { \big \|\, (\e\d_x)^\alpha A_{n,\sigma}\left(t,\cdot,\frac{\cdot}{\e}\right) \, \big \| 
}_{L^\infty} \leq C
\end{equation*}
holds true for all $\e\in(0,1)$, $|\alpha| \leq s$, and $t \in [0, T_*]$. 
Using the Gagliardo--Nirenberg inequality \eqref{GN}, we therefore require  
\begin{equation}\label{lem41eq1}
(\e\d_x)^\alpha A_{n,\sigma}\left(t,\cdot,\frac{\cdot}{\e} \right) \in 
H^{m}(\R^d)\quad\text{with ${m}>\frac{d}{2}$},
\end{equation}
for all $|\alpha|\le s$, i.e.
\begin{equation*}
A_{n,\sigma}\left(t,\cdot,\frac{\cdot}{\e}\right) \in 
H_\e^{m^*}(\R^d)\quad\text{with ${m^*}>s+\frac{d}{2}$}.
\end{equation*}
Thus, by \eqref{Ansigma}, we need $m^* =S-n+\tilde s>s+\frac{d}{2}$,
that is, either $S-n>|\beta|+\frac{d}{2}$ and $\tilde s\ge s-|\beta|$,
or $S-n\ge |\beta|$ and $\tilde s> s-|\beta|+\frac{d}{2}$, 
for all $|\beta|\le s$ and all $n=0,\ldots,N$. 
It turns out that in order to satisfy either of these conditions, the optimal
conditions we have to impose concerning the regularity of $A_{n,\sigma}$,
given in \eqref{Ansigma}, are 
$S-n>s+\frac{d}{2}$ and $\tilde s = K+2>s+\frac{d}{2}$, 
for $n=0,\ldots,N$, which directly yields the assumptions on $S$ and $K$
stated in the lemma.

The second estimate of the lemma then follows immediately by \eqref{lem41eq1}. In order to obtain the third estimate, 
having in mind the definitions \eqref{res} and \eqref{xn} with 
$A_{n,\sigma}\equiv0$ for $n\ge N+1$ and $a_{N,\vsi}\equiv0$ (\cf the note
after \eqref{intrans}), 
it is necessary and sufficient to assure additionally that
$\Delta a_{N-1,\vsi}\in H^s(\R^d)$, i.e., $a_{0,\vsi}\in H^{N+s+1}(\R^d)$.
This yields the additional condition $S\ge N+s+1$, and concludes the proof.
\end{proof}

\begin{remark}
Note that in order to derive the above given a-priori estimates for $u_N$,
the required regularity imposed on $\chi_\vsi$ is indeed independent of $N$. 
\end{remark}

We shall also need the following result on the linear time-evolution.

\begin{lemma} \label{lem3} For $\e \in (0,1)$ denote 
the unitary propagator corresponding to the linear Hamiltonian $H_{\rm per}^\e$, 
defined in \eqref{ham}, by
\[
U^\e(t) := \E^ {- \I H_{\rm per}^\e t/\e}.
\]
For $s\in [0,\infty)$ assume  
$\partial^\alpha V_\Gamma \in L^\infty (\YYY)$ for all $|\alpha| \leq
\max\{0,s{-}2\}$. 
Then there exists a $C_l>0$, such that 
\begin{equation}
\label{e:Ueps}
{\| \, U^\e(t) f^\e \, \|}_{H^s_\e} \leq {C_l}  \, {\|  f^\e  \|}_{H^s_\e}
\quad \text{ for all } t \in \R.
\end{equation}
\end{lemma}
\begin{proof} Recalling the definition of $H_{\rm per}^\e$, given in \eqref{ham}, 
we clearly have the basic $L^2$ estimate,
\[
{\|f^\e(t)\|}_{L^2} \equiv {\| U^\e(t) f^\e\|}_{L^2} = {\|f^\e \|}_{L^2},\quad 
 \text{ for all }  t \in \R,
\] 
since $H_{\rm per}^\e$ is self-adjoint.  Moreover, since $U^\e(t)$ obviously commutes
with (its generator) $H_{\rm per}^\e$, so does any power of the latter 
and we therefore obtain
\[
{\| (H_{\rm per}^\e)^{s/2} f^\e (t)\|}_{L^2} = {\| (H_{\rm per}^\e)^{s/2} f^\e \|}_{L^2}, \quad 
  \text{ for all } s, t\in \R.
\]
Without loss of generality we assume here that $H_{\rm per}^\e \geq 1$, 
otherwise we may add a constant independent of $\e\in (0,1)$. 

For $s=1$ this is nothing but energy conservation. Using the conservation of
the $L^2$ norm and the condition $V_\Gamma \in L^\infty(\YYY)$ we immediately
obtain the desired result \eqref{e:Ueps} for $s=1$ with $C_l= (
1{+}4\|V_\Gamma\|_{\infty})^{1/2}$. 

For general $s\in \N{\setminus}\{1\}$, we use integration by parts and 
the assumed regularity of $V_\Gamma$ to find 
 $\e$-independent constants $C_1,C_2 >0$ such that
\[
C_1 {\|(H_{\rm per}^\e)^{s/2}f^\e(t) \|}_{L^2} \leq \, {\|f^\e(t) \|}_{H^s_\e} \, \leq
\, C_2 {\|(H_{\rm per}^\e)^{s/2}f^\e (t) \|}_{L^2},
\]
\ie we have equivalence of the norms uniformly in $\e$. For general $s \geq 0$ 
the same holds true by interpolation. In summary we obtain boundedness of the
unitary group $U^\e(t)$ in all $H^s_\e$ 
and the assertion of the lemma is proved.
\end{proof}

\subsection{Stability of the approximation} 
\label{ss4.2}

First let us recall the following Moser-type lemma, \cf \cite[Lemma~8.1]{Ra},
which we shall use in the proof of Theorem \ref{th1} below.
\begin{lemma}\label{lem4}
Let $R>0$, $s \in [0,\infty)$, and $\mathcal N(z)=\kappa |z|^{2}z$ with $\kappa\in\R$. 
Then there exists a $C_s=C_s(R,s,d,\kappa)>0$ such that if 
\begin{equation*}
{\left\| (\e\d)^\alpha f \, \right\|}_{L^\infty} \leq R, 
\ \forall \, |\alpha| \leq s \ \text{and} \ 
\displaystyle {\left\| \, g \, \right\|}_{L^\infty} \leq R, 
\end{equation*}
then 
\begin{equation*}
{\left\| \, \mathcal N(f + g)
- \mathcal N(f) \, \right\|}_{H_\e^s} 
\leq C_s  {\left\| \, g \, \right\|}_{H_\e^s}\, . 
\end{equation*}
\end{lemma}
\begin{proof} The proof can be found in \cite{Ra} for $s \in \N$ and follows by interpolation for general $s \in [0,\infty)$.
\end{proof}

With the above results at hand, we are now able to establish our main 
result, which justifies rigorously the validity of the amplitude equations
\eqref{systemIntro}, describing the macroscopic dynamics of $M$ 
modulated pulses for a closed mode system of order $\Lambda=2N+1$ (with $N$ depending on $d$). 

\begin{theorem}\label{th1}
For $d\in\N$ choose  $s, S \in[0,\infty)$ and $N\in \N$ such that 
$N,s>\frac{d}{2}$ and $S>N+s+\frac{d}{2}$, or, if $d=1$, $S\ge N+s+1$.  
Let $K=\max\{0,s+\frac{d}{2}-2\}$ and  assume 
$\partial^\alpha V_\Gamma \in L^\infty (\YYY)$ for all $|\alpha| \leq K$. 
Moreover, let the finite system of modes $\mathcal
S=\{\mu_1,\ldots,\mu_M: M\in\N\}$ be closed of order $\Lambda=2N+1$. 

Then for any solution $(a_1,\ldots,a_M) \in C^0( [0,T),H^S(\R^d))^M$ 
of the amplitude equations \eqref{systemIntro}, 
and any $t_*\in(0,T)$, $\beta\in(\frac{d}{2},N]$, $c>0$,
there exist an $\e_0\in(0,1)$ and a $C>0$, 
such that for all $\e\in(0,\e_0)$,
the approximate solution $u_{N}^\e$, constructed above, 
and any exact solution $u^\e$ of \eqref{nls}
with 
\[{\| \, u^\e (0,\cdot)-u_{N-1}^\e (0,\cdot) \, \|}_{H_\e^s} \leq c \e^\beta,\]
satisfy 
\[{\| \, u^\e (t,\cdot)-u_{N-1}^\e (t,\cdot) \, \|}_{H_\e^s} \leq C \e^\beta
\quad \text{for all  }t \in [0, t_*].
\]
\end{theorem}
\begin{remark}
  Note, that the approximate solution $u_N^\e$ contains terms up to order
  $\O(\e^N)$, whereas the above given error estimates include only
  $u^\e_{N-1}$.  The reason for this is that for our proof we need to work 
  with $u_N^\e$ but still the obtained error is only of the order
  $\O(\e^\beta)$ with $\beta\le N$. We can therefore eventually neglect the
  last term in $u_N^\e$, or, loosely speaking, we can move it to
  the right hand side of the above given estimate.
\end{remark}
\begin{proof} We write the exact solution of \eqref{nls} in the form
\[
u^\e(t,x) = u_N^\e(t,x) + \e^{\beta } w^\e(t,x)
\]
and denote $\varrho^\e(t):={ \| \, w^\e (t) \, \|}_{H_\e^s}$. Then, 
clearly, $\varrho^\e(0)\leq c$, 
and we will show that there exist $C>0$, $\e_0 \in (0,1)$, 
such that, for all $\e \in (0, \e_0]$, it holds $\varrho^\e(t) \leq C$ 
for $t \in [0, t_*]$. Inserting this ansatz into \eqref{nls}, written as 
\begin{equation*}
\I \partial _t u^\e =  -\frac{\e}{2}\Delta u^\e +
\frac{1}{\e} \, V_{\Gamma}\left(\frac{x}{\e}\right)u^\e + 
\mathcal N (u^\e),\quad \text{where} \ \mathcal N(z)=\kappa\, |z|^{2}z,
\end{equation*} 
and applying Duhamel's formula we get 
\begin{align*}
w^\e(t) = & \ U^\e(t) w^\e(0) + {\e^{-\beta}}\int_0^t\, U^\e(t-\tau) 
\left( \mathcal N (u_N^\e(\tau) + \e^{\beta } w^\e(\tau) ) - 
\mathcal N(u_N^\e(\tau))\right)\, \D \tau \\
& \ - {\e^{-(\beta+1)}} \int_0^t  \, U^\e(t-\tau) \, \text{res}(u^\e_N(\tau)) 
\, \D \tau, 
\end{align*}
where $\text{res} (u_N^\e)$ is defined in equation \eqref{res}. 
Hence, using Lemma \ref{lem2-neu} and Lemma \ref{lem3} to estimate the
residual and the linear semi-group, respectively, we obtain 
\begin{equation*}
\varrho^\e(t)\leq \ C_l c  + C_l C_r  \e^{N-\beta} t
+ \, 
{\e^{-\beta}} 
\int_0^t \, C_l 
{\left \| \, \mathcal N (u_N^\e(\tau) + \e^{\beta } w^\e(\tau) ) - 
\mathcal N(u_N^\e(\tau))  \right \|}_{H^s_\e} \, \D \tau  ,
\end{equation*}
since $\varrho^\e(0)\leq c$, by assumption. Using $N\ge \beta$ and 
$\e\in(0,1)$, we consequently obtain 
\begin{equation*}
\varrho^\e(t)\leq \ C_l (c  + C_r t_*) 
+ 
\e^{-\beta}
\int_0^t  \, C_l 
{\left \| \, \mathcal N (u_N^\e(\tau) + 
\e^\beta 
w^\e(\tau) ) - 
\mathcal N(u_N^\e(\tau)) \right  \|}_{H^s_\e} \,\D \tau ,
\end{equation*}
for $t\le t_*$. 

Now, we set $C:=C_l (c + C_r t_*){\mathrm e}^{ C_l C_s t_*}$ and choose a 
$D>\max\{c, C\}$. Then, since $D>c \geq \varrho^\e(0)$ 
and $\varrho^\e(t)$ is continuous, there exists, for every 
$\e\in(0,1)$, a positive time $t^\e_{D}>0$, such that $\varrho^\e(t)\le D$ for $t\le t^\e_{D}$.

The Gagliardo--Nirenberg inequality \eqref{GN} yields, for $s>d/2$, that   
\[
{\| \, \e^{\beta} w^\e (t) \, \|}_{L^\infty} \leq \e^{\beta-d/2} C_\infty D
\qquad\text{for $t\le t^\e_{D}$}.
\]
Hence, using $\beta-d/2>0$ there exists an $\e_0\in(0,1)$, such
that  
\[
{\| \, \e^{\beta } w^\e (t) \, \|}_{L^\infty} \leq C_a, \quad \text{for $\e\in(0,\e_0]$ and $t\le t^\e_{D}$}.
\]
Moreover, by Lemma \ref{lem2-neu} we have 
${\left\| (\e\d)^\alpha u_N^\e (t)\, \right\|}_{L^\infty} \leq C_a$ 
for $|\alpha|\le s$, $\e\in(0,1)$, and $t<T$.
Thus, we can apply Lemma \ref{lem4} (with $R=C_a$) in order to estimate 
the nonlinear term and obtain
\begin{equation*}
\varrho^\e(t) \leq \ C_l( c  + C_r t_*)
+ C_l C_s 
\int_0^t  \,\varrho^\e(\tau) \, \D \tau,
\quad \text{for $\e\in(0,\e_0]$ and $t\le t^\e_{D}$}.
\end{equation*}
Gronwall's lemma then yields
\begin{equation*}
\varrho^\e(t) \leq \ C_l( c  + C_r t_*) {\mathrm e}^{ C_l C_s t}
\leq C, \quad \text{for $\e\in(0,\e_0]$ and 
%$t\le t^\e_{D}$
% hier stand bei mir sowieso t_* ... ``und das war gut so'' !!!
$t\le t_*$}.
\end{equation*}
Since $C<D$, we conclude that the assumptions needed in order to apply 
Lemma \ref{lem4} are fulfilled for all $\e\in(0,\e_0]$ and all $t\le t_*$, 
%$t\le t_*\e^s$, 
that is 
$t^\e_{D}\ge t_*$.
%$t^\e_{D}\ge t_*\e^s$.
Thus the above given estimate proves that 
\[
{\| \, u^\e (t,\cdot) - u_{N}^\e (t,\cdot) \, \|}_{H_\e^s}
={\mathcal O}(\e^{\beta}), 
\quad \text{for $t\in[0,t_*]$}.
\]
However, since 
$N\ge\beta$ and
\[
{\| \, u_N^\e (t,\cdot) - u_{N-1}^\e (t,\cdot) \, \|}_{H_\e^s}
={\mathcal O}(\e^{N}),\quad  \text{for $t \in [0,t_*]$}, 
\]
we can finally replace $u_N^\e (t,\cdot)$ by 
$u_{N-1}^\e (t,\cdot)$ in our stability result, which 
consequently proves the assertion of the theorem.
\end{proof}
We want to stress that our stability result is an advancement when compared to 
the result of \cite{CMS}, in the sense that our asymptotic estimates 
no longer suffer from a loss of accuracy in powers of $\e$ (as has been 
the case in the cited work). We infer in particular, that our nonlinear setting 
allows for the same kind of stability result as one would expect in the linear 
case.
 
Theorem \ref{th1} also shows that no blow-up can occur 
in the exact solution $u^\e$ within the interval $t \in [0, T)$, determined 
by the existence time of \eqref{systemIntro}. Hence, if the system
\eqref{systemIntro} indeed admits global-in-time solutions we deduce that the 
solutions $u^\e$, starting close to such modulated pulses, exist for arbitrary 
long times (but blow-up may occur after the modulational structure is lost 
on longer time scales.)

Finally note that for $d=1$ the assumption $N> d/2$ 
and the required closure of order $\Lambda=2N+1$ imply that, at least, 
$N=1$ and thus $\Lambda = 3$, whereas for $d=3$ spatial dimensions 
we need at least $N=2$ and hence $\Lambda = 5$.
This increase of the required $\Lambda$ for higher 
spatial dimensions $d$ can be relaxed though, as we shall show in the following 
section.

\section{The case of higher-order resonances}
\label{highres}

In Section \ref{formal}, we derived the approximate solution $u^\e_N$ 
under the assumption that $\mathcal S$ satisfies a closure condition of 
sufficient high order $\Lambda$, \ie we required $\Lambda = 2N+1$ and $N>d/2$. 
We shall show now that this closure condition can be significantly relaxed
if we generalize the ansatz \eqref{wkb1}, \eqref{wkb2} slightly by allowing
for a larger set of modes $\mathcal G_{2N+1}$. To this end we consider the new ansatz
\begin{equation}
\label{e:ansatz}
u_N^\e (t,x)= \sum_{n=0}^N \sum_{\sigma \in \mathcal G_{2n+1}} 
\e^n  \, A_{n, \sigma}(t,x,y) \be_{\sigma}\left(\tau,y\right),
\end{equation}
where the set of modes $\mathcal G_{2N+1}$ is then defined inductively as follows: 
Again we start from $\mathcal
S=\{(k_1,\ell_1),...,(k_M, \ell_M)\}$, which is now assumed to be \emph{only} closed of order
$\Lambda =3$ and we set $\mathcal G_1=\mathcal G^{(1)}_\mathcal S$. 
However, for $n\geq 2$ we will allow the sets $\mathcal G_{2n+1}$ to be larger than
$\mathcal G^{(2n+1)}_\mathcal S$. The reason for this enlargement is motivated by the
fact, that 
\[
\widetilde {\mathcal G}^{(2n+1)}_{\mathcal S} = \mathcal G^{(2n+1)}_{\mathcal S}\cap 
\Big( \mathcal G \setminus \mathcal G^{(1)}_{\mathcal S} \Big)
\]
may be nonempty for some $n\geq 2$. Then, the corresponding equation
\eqref{allg} might not be solvable. To circumvent this problem, we set 
$$
\mathcal G_{2n-1}=\mathcal G^{(2n-1)}_{\mathcal S} \cup \widetilde {\mathcal G}_{\mathcal S}^{(2n+1)}.
$$ 
At stage $n{-}1$ the corresponding $A_{n-1,\sigma}$ with $\sigma \in 
\widetilde {\mathcal G}^{(2n+1)}_{\mathcal S}$ take the usual product form $a_{n-1,\sigma}(t,x)\chi_\sigma(y)$ with
$\chi_\sigma  \in \mathrm{ker}L^\sigma_0 \subset H^2(\YYY)$ and free
coefficients $a_{n-1,\sigma}$.  Then, in step $n$ the
corresponding equations for $\sigma\in \widetilde {\mathcal G}_{\mathcal S}^{(2n+1)}$ can be solved as
in Case II of Subsection \ref{ss3.1}, \ie we obtain a linear transport equation for the free
coefficients from the solvability condition for the previous step. 

To make this more precise, we associate with the sequence
$\mathcal G_1,\mathcal G_3,\ldots, \mathcal G_{2N+1}$  another sequence $\check {\mathcal G}_3,\check {\mathcal G}_5,\ldots, \check {\mathcal G}_{2N+1}$ as follows
\begin{equation}\label{e:assoc-seq}
\check {\mathcal G}_{2n+1} = \bigcup_{n_1,n_2,n_3=0,...,n-1 \atop
 n_1+n_2+n_3=n-1} \Big( \mathcal G_{2n_1+1} -\mathcal G_{2n_2+1} +\mathcal G_{2n_3+1}\Big) .
\end{equation}
Hence, $\check {\mathcal G}_{2n+1}$ contains all modes of order $\O(\e^n)$ 
that can be generated by the cubic nonlinearity $|u|^2 u$ from the already
given terms of lower order. 

\begin{definition}\label{d:weak-res}
A mode set  $\mathcal S=\{(k_1,\ell_1),...,(k_M, \ell_M)\}$ is called
\emph{weakly closed of order $2N{+}1$}, 
if there exists a sequence ${\mathcal G}_1\subset
{\mathcal G}_3 \subset \cdots \subset {\mathcal G}_{2N+1} \subset \mathcal B\times \R$ such that the
following conditions hold:
\begin{equation*}
\begin{aligned}
\text{(i) }\ & {\mathcal G}_{2n+1} \text{ is finite for }n=0,1,...,N; \\
\text{(ii) }\ & {\mathcal G}_1 ={\mathcal G}^{(1)}_\mathcal S=\Sigma(\mathcal S); \\
\text{(iii) }\ & \check {\mathcal G}_{2n+1} \subset {\mathcal G}_{2n+1} \text{ for } n=1,...,N; \\
\text{(iv) }\ &\text{if } \sigma \in {\mathcal G}_{2n+1}\cap {\mathcal G} \text{ with }n\in \N, 
  \text{ then either }  \sigma \in {\mathcal G}_{2n-1} \text{ or } 
   \sigma \not\in \check {\mathcal G}_{2n+1};\\
\text{ (v) }\ & \text{for each }\sigma% =(k,E)
   \in {\mathcal G}_{2N+1}\cap {\mathcal G} \text{ the group
  velocity } \vartheta_\sigma =\nabla_k E_\ell(k) \text{ exists},
\end{aligned}
\end{equation*}
where $\check {\mathcal G}_3\subset \cdots \subset \check {\mathcal G}_{2N+1}$ is the sequence 
associated with ${\mathcal G}_1,...,{\mathcal G}_{2N+1}$ according to \eqref{e:assoc-seq}.  
\end{definition}

Condition (iv) means that any occurring resonant mode
must either occur already in an earlier step and hence has a corresponding
free coefficient or it appears for the first time but it is not yet generated
by the nonlinear interaction. This will become clearer in the following examples.

\begin{example}\label{ex:weak1}
We first show that a system which is closed of order $2N{+}1$ is also weakly
closed of order $2N{+}1$. For this we simply set
${\mathcal G}_{2n+1}={\mathcal G}^{(2n+1)}_\mathcal S$. By construction, the associated sequence is
$\check {\mathcal G}_{2n+1}={\mathcal G}_{2n+1}$ and the conditions (i) to (iii) hold
immediately. Moreover, (iv) follows since ${\mathcal G}_{2n+1}\cap {\mathcal G}= {\mathcal G}_1={\mathcal G}^{(1)}_\mathcal
S$.     
\end{example}

\begin{example}\label{ex:weak2}
We now show that a weak closure of order $\Lambda$ is in general a weaker condition than 
closure of order $\Lambda$ is. For this
consider the case $\mathcal S=\{(k_1,\ell_1),(k_2,\ell_2)\}$ such that
${\mathcal G}_1={\mathcal G}^{(1)}_\mathcal S=\{\sigma_1,\sigma_2 \}\subset {\mathcal G}$ with $\sigma_j =
(k_j,E_{\ell_j}(k_j)) $. An easy induction argument gives 
\[
{\mathcal G}^{(2n+1)}_\mathcal S=\{\: (n{+ }1{-}j)\sigma_1 +(j{-}n)\sigma_2 \: ; \: 
j=0,1,\ldots, 2n+1\:\}
\]
Now assume that ${\mathcal G}\cap {\mathcal G}^{(2 N^*+1)}_\mathcal S\subset \{ \sigma_1,\sigma_2 ,
3\sigma_1{-}2\sigma_2\} $ for some $N^*\geq 2$, \ie we have a closure of order 3
but not of order 5. 

We claim that a weak closure of order $2N{+}1$ still holds for any $N$ such that 
$3N \leq  2 N^*$. For $n=0,\ldots, N$ we let  
\[
{\mathcal G}_{2n+1} = \{\: j\sigma_1 +(1{-}j)\sigma_2 \: ; \: 
j=-\beta_n,\ldots, \alpha_n\:\},
\]
with $\alpha_n= [3(n{+}1)/2] $ and $\beta_n=[3n/2]$, where $[\,\cdot\,]$ denotes
the integer part. By construction of the sequences $(\alpha_n,\beta_n)_{n}$
we thus have ${\mathcal G}_3=\check {\mathcal G}_3\cup\{3\sigma_1{-}2\sigma_2\}$ and
${\mathcal G}_{2n+1}=\check {\mathcal G}_{2n+1}$ for $n=2,\ldots,N$. Hence conditions (i) to (iii)
of Definition \ref{d:weak-res} are fulfilled. Condition (iv) also holds 
since ${\mathcal G}_{2n+1}\cap {\mathcal G}\subset {\mathcal G}_1\cup\{3\sigma_1{-}2\sigma_2\}$.   
\end{example}

Repeating the construction of Section \ref{ss3.2} analogously we obtain 
the following result. 

\begin{lemma}\label{l:weak-ansatz}
Let $\mathcal S=\{(k_1,\ell_1),\ldots, (k_M,\ell_M)\}$ be a set of modes that
is closed of order $\Lambda=3$ and weakly closed of order $\Lambda=2N{+}1$ 
for some $N\in \N$. Then an approximation 
\begin{equation}\label{approxGeneral}
u^\e_N(t,x) = \sum_{n=0}^N \e^n \sum_{\sigma\in {\mathcal G}_{2n+1}} A_{n,\sigma}(t,x,y)
\be_\sigma 
\end{equation}
can be constructed as in Sections \ref{ss3.1} and \ref{ss3.2}, 
where ${\mathcal G}_1\subset \cdots {\mathcal G}_{2N+1}$ is the sequence guaranteed by Definition
\ref{d:weak-res}.  
\end{lemma}

Hence, via the same steps as in Section \ref{just} 
one can easily obtain a justification of the amplitude equations \eqref{systemIntro} 
under these relaxed resonance conditions.
\begin{corollary}
Let $d,N\in\N$ and ${\mathcal S}$ be a mode system that is closed of order 
$\Lambda=3$ and weakly closed of order $\Lambda=2N+1$.
Then, under the same assumptions as before the statement of Theorem \ref{th1} 
holds analogously with the approximate solution given by \eqref{approxGeneral}.
\end{corollary}

\appendix

\section{Hamiltonian structure and conservation laws  
of the amplitude equations}

Recalling the amplitude equations in the general form
\eqref{systemIntro}, \ie  
\begin{equation}
\label{e4:system}
\i \partial_t a_m =-\i \vartheta_m\cdot \nabla_x a_m + \sum^M_{ {p,q,r =1:} \atop {\Sigma(\mu_p, \mu_q, \mu_r) =\Sigma(\mu_m)} }
\kappa_{(p,q,r,m)} a_p \ol a_q a_r, 
\end{equation}
for $m=1,\ldots,M$, we want to highlight the Hamiltonian structure of this
equation as well as to deduce the conserved quantities which
follow from the specific structure of the resonances. For a general method on
how to derive the \emph{reduced Hamiltonian structure} and the first
integrals for \eqref{e4:system} we refer to \cite{GHM,GiHeMi06?LHTS}.  Here we
just collect the results that are obtained with the general theory developed
there.

We first note that \eqref{e4:system} is a Hamiltonian system generated by 
\[
\mathcal H^\mathrm{red} (a) = \int_{\R^d} \sum_{m=1}^M \Im\left(\ol a_m \:
  \vartheta_m \cdot \nabla a_m \right) + \sum^M_{ {p,q,r =1:} \atop
  {\Sigma(\mu_p, \mu_q, \mu_r) =\Sigma(\mu_m)} }
\frac{\kappa_{(p,q,r,m)}}{2}\; a_p \ol a_q a_r \ol a_m \; \D x,
\]    
and the symplectic two-form $\i = \sqrt{-1}$. Hence, \eqref{e4:system} takes
the form
\[
\i \d_t a_m =\delta_{\ol a_m} \mathcal H^\mathrm{red}(a) .
\]
Note that the reduced Hamiltonian $\mathcal H^\mathrm{red} (a)$ 
is not obtained as the lowest order
expansion of the original Hamiltonian $\mathcal H^\e(u^\e)$ of the full system, which reads
\[
\mathcal H^\e (u^\e) = \int_{\R^d} \frac{\e^2}2 |\nabla u  |^2 +
 V_\Gamma\left(\frac{x}{\e}\right)|u^\e|^2 +\frac{\e \kappa}2 |u|^4 \;\D x  .
\] 
Indeed, inserting the ansatz $u^\e=u^\e_N$, as given in \eqref{uNfinal}, into $\mathcal H^\e(u^\e)$ 
we find, as $\e \to 0$, that
\[
\mathcal H^\e(u^\e_N)= \mathcal I(a)+ \O(\e), \quad \text{with }
\mathcal I(a)=\int_{\R^d} \sum_{m=1}^M E_{\ell_m}(k_m)
|a_m|^2 \:\D x .
\]
Even though not a Hamiltonian, 
$\mathcal I(a)$ clearly is a conserved quantity for the amplitude system
\eqref{e4:system}. 

However, the energy levels $E_{\ell_m}(k_m) \equiv \omega_m$ do not occur
explicitly in the amplitude system \eqref{e4:system}. Hence, any
choice of $\widetilde \omega_m \in \R$ that is compatible with the resonance
conditions \eqref{rcon} leads to additional conserved quantities. More
precisely, if $\widetilde \omega_1,\ldots, \widetilde \omega_M$ are chosen
such that for any $p,q,r,m\in \{1,\ldots,M\}$ the resonance conditions
\eqref{rcon} imply the identity
\[
\widetilde \omega_p- \widetilde \omega_q+ \widetilde \omega_r= \widetilde \omega_m,
\] 
then
\[
\widetilde{\mathcal I}(a)= \int_{\R^d} \sum_{m=1}^M \widetilde \omega_m |a_m|^2
\;\D x
\]
defines a first integral. 

A particular choice is $\widetilde \omega_1=\cdots= \widetilde \omega_M=1$, which leads to the trivial fact that the
$L^2$ norm is preserved. The latter can of course also be deduced from
the fact that the $L^2$ norm was preserved in the original problem for $u^\e$ or
from the fact that the mode system is invariant under the phase shifts 
$(a_1,\ldots,a_M) \mapsto (\E^{\i \alpha}a_1,\ldots,\E^{\i \alpha}a_M)$, with
$\alpha \in \R$.  
Similarly, $\widetilde{\mathcal I}(a)$ can be understood as
a first integral with respect to suitably chosen phase shifts that are
compatible with the resonance structure, \ie 
\[
\widetilde T_\alpha: (a_1,\ldots,a_M) \mapsto (\E^{\i \widetilde \omega_1
  \alpha}a_1,\ldots, \E^{\i \widetilde \omega_M \alpha}a_M),.
\]
Finally, it should be mentioned that the mode system \eqref{e4:system} is
translation invariant. This provides first integrals associated with the
translation operators in the coordinate directions 
\[
\mathcal I^\mathrm{trans}_\theta (a)= \int_{\R^d}\sum_{m=1}^M  \Im \left(\ol
  a_m\; \theta\cdot \nabla a_m  \right) \D x, \quad {\theta \in \R^d}.
\]

So far, we are not able to show that these conserved quantities are enough to
provide a global existence result. Note in particular, that for the original
problem the case $\kappa>0$ leads to an energy that is definite and allows us
to conclude global existence for the nonlinear Schr\"odinger equation
\eqref{nls}. For the mode system \eqref{e4:system} the sign of $\kappa$ is no
longer helpful, since the term involving derivatives is indefinite. Note also 
that global existence for the mode system cannot be inferred from global 
existence of \eqref{nls}, since we can not expect the solutions $u^\e$ to 
remain in the form of modulated pulses for all times. 

\bibliographystyle{amsplain}

\end{document}